\documentclass{article}
\usepackage{amsmath, amscd, amsthm, amscd, amsfonts, amssymb, graphicx, color, soul, enumerate, xcolor,  mathrsfs, latexsym}
\usepackage{pstricks,pst-node,pst-tree}
\usepackage[bookmarksnumbered, colorlinks, plainpages,backref]{hyperref}
\hypersetup{colorlinks=true,linkcolor=red, anchorcolor=green, citecolor=cyan, urlcolor=red, filecolor=magenta, pdftoolbar=true}

\theoremstyle{plain}
\newtheorem{theorem}{Theorem}
\newtheorem{lemma}[theorem]{Lemma}

\newtheorem{proposition}[theorem]{Proposition}

\newtheorem{procedure}[theorem]{Procedure}

\theoremstyle{definition}
\newtheorem{definition}[theorem]{Definition}
\newtheorem{example}[theorem]{Example}

\theoremstyle{remark}
\newtheorem{remark}[theorem]{Remark}

\author{
Mohammad Hadi Shekarriz, Madjid Mirzavaziri\\
Department of Pure Mathematics, \\Ferdowsi University of Mashhad,\\ P.O. Box 1159, Mashhad 91775, Iran\\
mh.shekarriz@mail.um.ac.ir, mirzavaziri@um.ac.ir
\and 
and\\
Kamyar Mirzavaziri\\
National Organization for Development of Exceptional Talents\\
Mashhad, Iran\\
mirzavaziri@gmail.com}

\title{A Characterization For 2-Self-Centered Graphs}

\begin{document}
\maketitle

\begin{abstract}  A Graph is called 2-self-centered if its diameter and radius both equal to 2. In this paper, we begin characterizing these graphs by characterizing edge-maximal 2-self-centered graphs via their complements. Then we split characterizing edge-minimal 2-self-centered graphs into two cases. First, we characterize edge-minimal 2-self-centered graphs without triangles by introducing \emph{specialized bi-independent covering (SBIC)} and a structure named \emph{generalized complete bipartite graph (GCBG)}. Then, we complete characterization by characterizing edge-minimal 2-self-centered graphs with some triangles. Hence, the main characterization is done since a graph is 2-self-centered if and only if it is a spanning subgraph of some edge-maximal 2-self-centered graphs and, at the same time, it is a spanning supergraph of some edge-minimal 2-self-centered graphs.
\end{abstract}

\section{Introduction}
Let $G=(V,E)$ be a connected finite simple graph. For $u,v\in V$ the \textit{distance} of $u$ and $v$, denoted by $d_G(u,v)$ or $d(u,v)$, is the length of a shortest path between $u$ and $v$. The \textit{eccentricity} of a vertex $v$, ${\rm ecc}(v)$, is $\max\{d(u,v): u\in V\}$. The maximum and minimum eccentricity of vertices of $G$ are called \textit{diameter} and \textit{radius} of $G$ and are denoted by ${\rm diam}(G)$ and ${\rm rad}(G)$, respectively. \emph{Center} of a graph $G$ is the subgraph induced by vertices with eccentricity ${\rm rad}(G)$. A graph is called self-centered if it is equal to its center, or equivalently, its diameter equals its radius.

A graph $G$ is called \textit{$k$-self-centered} if ${\rm diam}(G)={\rm rad}(G)=k$. The terminology \textit{$k$-equi-eccentric graph} is also used by some authors. For studies on these graphs see \cite{Buc, Buc2, Aki, Bal, Neg} and \cite{Kla}.

Clearly, a graph $G$ is 1-self-centered if and only if $G$ is a complete graph. In this paper, we try to characterize 2-self-centered graphs. An edge-maximal 2-self-centered graph can be easily characterized via a condition on its complement. We do this in section 2, using a lemma which gives a necessary and sufficient condition for a graph to be 2-self-centered. For edge-minimal 2-self-centered graphs we need to divide the discussion into two cases: triangle-free or not. We do these in section 3. Then, the characterization is done in sight of the following theorem. 

\begin{theorem}
\label{Char}
A finite graph $G$ is 2-self-centered if and only if there is an edge-minimal 2-self-centered graph $G'$ and an edge-maximal 2-self-centered graph $G''$ such that $G'$ is a spanning subgraph of $G$ while $G$ is itself a spanning subgraph of $G''$.
\end{theorem}

\begin{proof}
The proof is clear. Note that $G'\subseteq G$ implies ${\rm rad}(G)\geqslant{\rm rad}(G')=2$ and $G\subseteq G''$ implies ${\rm diam}(G)\leqslant{\rm diam}(G'')=2$.
\end{proof}

Throughout this paper $G$ is a connected finite simple graph and its \textit{complement} is denoted by $\overline{G}$. If $G$ is a graph and $e$ is an edge in $G$, then $G\setminus e$ is the graph obtained from $G$ by omitting $e$. Moreover, the graph obtained by adding an edge $e\notin E(G)$ to $G$ is denoted by $G+ e$. Whenever two vertices $u$ and $v$ are adjacent, we might write $u \sim v$. For concepts and notations of graph theory, the reader is referred to \cite{Gro}. 

\section{edge-maximal 2-Self-centered Graphs}

In this section, we present a characterization for edge-maximal 2-self-centered graphs. The following Lemma is not only essential to do so, but it is also going to be used all over this paper.

\begin{lemma}\label{2self}
Let $G=(V,E)$ be a graph with $n$ vertices. Then
$G$ is 2-self-centered if and only if the following two conditions are true:
\begin{itemize}
\item[(i)] $2\leqslant \deg(v)\leqslant n-2$ for all $v\in V$;
\item[(ii)] for each $u,v\in V$ with $uv\notin E$ there is a $w\in V$ such that $uw,wv\in E$.
\end{itemize}
\end{lemma}
\begin{proof}
The proof is obvious. Note that if $G$ has a vertex $v$ with $\deg(v)=n-1$ then ${\rm rad}(G)=1$ and if there is a vertex $u$ with $\deg(u)=1$ then its neighbour should be adjacent to any vertex of $G$, since otherwise ${\rm ecc}(u)>2$.
\end{proof}

\begin{remark}
\label{remark}
If we show that for a graph $G$ item (ii) of Lemma \ref{2self} holds and no vertex is adjacent to all vertices, then we can deduce that no vertex has degree 1 and therefore $G$ is 2-self-centered.
\end{remark}

A 2-self-centered graph $G$ is said to be \emph{edge-maximal} if there are no non-adjacent $u,v \in V(G)$ such that $G + uv$ is 2-self-centered. The following theorem is a characterization for edge-maximal 2-self-centered graphs.

\begin{theorem}
Let $G$ be a 2-self-centered graph. Then $G$ is edge-maximal if and only if $\overline{G}$ is disconnected and each connected component of $\overline{G}$ is a star with at least two vertices. 
\end{theorem}
\begin{proof}
Let $H_1,\ldots,H_r$ be the connected components of $\overline{G}$, where $r$ is a positive integer. At first, note that each $H_i$ should be a tree with at least two vertices. To see this, if $H_i$ has only one vertex $v$ then the degree of $v$ in $\overline{G}$ is zero and thus its degree should be $n-1$ in $G$ which contradicts to (i) of Lemma \ref{2self}. Furthermore, if the connected component $H_i$ is not a tree then there is an edge $e$ with end vertices $u_0$ and $v_0$ in $H_i$ which is not a cut edge. Let $H=G+ e$. Since $G$ is edge-maximal, $H$ cannot be 2-self-centered. Using Lemma \ref{2self}, we can deduce that the degree of $u_0$ or $v_0$ in $G$ must be $n-2$. This means that the degree of $u_0$ or $v_0$ in $\overline{G}$ is 1 and consequently $e$ is a cut edge, a contradiction.

Now, we show that each connected component $H_i$ is a star. Let $u$ be a vertex with maximum degree $k$ in $H_i$. If $k=1$ then $H_i$ is $K_{1,1}$. Let $k\geqslant2$. If $H_i$ is not $K_{1,k}$ then one of the neighbours of $u$, say $v$, has a neighbour $w\neq u$. Let $e'$ be the edge between $u$ and $v$ in $\overline{G}$ and $H'=G+ e'$. Since $G$ is edge-maximal, $H'$ cannot be 2-self-centered. Using Lemma \ref{2self}, we can again deduce that the degree of $u$ or $v$ in $G$ should be $n-2$. This means that the degree of $u$ or $v$ in $\overline{G}$ is 1; which is a contradiction.

Conversely, suppose that $\overline{G}$ is a disconnected graph whose connected components are all stars, each of which has at least two vertices. Then, $2\leqslant\deg(v)\leqslant n-2$ for all  $v \in V(G)$ and whenever  $u$ and $v$ are two non-adjacent vertices of $G$, there must be a $w \in V(G)$ such that $u$ and $v$ are both adjacent to $w$. Therefore, by Lemma \ref{2self} $G$ is a 2-self-centered graph. Moreover, since every connected component of $\overline{G}$ is a star with at least two vertices, adding an edge between two non-adjacent vertices in $G$ makes the complement to have a singleton as a connected component, which means that the resulted graph is not 2-self-centered. 
\end{proof}

\section{Edge-Minimal 2-Self-centered Graphs}
A 2-self-centered graph $G$ is said to be \emph{edge-minimal} if for each $e \in E(G)$, $G\setminus e$ is not a 2-self-centered graph. In this section, we determine all edge-minimal 2-self-centered graphs. To do so, let at first suppose that $\overline{G}$ is disconnected.
 
\begin{proposition}\label{bipartite}
Let $G$ be a graph. Then $G$ is an edge-minimal 2-self-centered graph such that $\overline{G}$ is disconnected if and only if it is the complete bipartite graph $K_{k,\ell}$ for some $k,\ell\geqslant2$.
\end{proposition}
\begin{proof}
Let $H_1,\ldots,H_r$ be the connected components of $\overline{G}$, where $r\geqslant2$. At first we prove that each $H_i$ is a clique in $\overline{G}$, or in another word, each $H_i$ is an independent set in $G$. Let $e$ be an edge in $G$ between two vertices $u$ and $v$ of $H_i$. If $H=G\setminus e$, then edge minimality of $G$ implies that $H$ cannot be 2-self-centered. 

Let $u'$ and $v'$ be two non-adjacent vertices of $H$. Then $u'$ and $v'$ are belonged to a connected component $H_j$ of $\overline{G}$. Let $w'$ be any vertex of $H_{j'}$, where $j'\neq j$. Thus $u'w',w'v'\in E(H)$. This shows that $H$ satisfies part (ii) of Lemma \ref{2self}.

Since $H$ is not 2-self-centered, Lemma \ref{2self} implies that the degree of $u$ or $v$ in $H$ is 1. Let the degree of $u$ in $H$ be 1. Thus $u$ has a neighbour $w$ in $H$. This implies that all other vertices of $G$ are in $H_i$. We knew that $v$ is also in $H_i$. Thus $H_i$ contains all vertices except $w$ and $w$ is itself a component. Hence, the degree of $w$ in $G$ is $n-1$ which contradicts to Lemma \ref{2self}.

Now we show that $r=2$. Let $r\geqslant3$. Choose $x,y$ and $z$ in three different components. Let $e=xy$ and $H=G\setminus e$. Due to the existence of $z$, $H$ is clearly 2-self-centered which contradicts to the edge-minimality of $G$.

Conversely, the complete bipartite graph $K_{k,\ell}$ for $k,\ell\geqslant2$ is an edge-minimal 2-self-centered graph such that its complement is disconnected. 
\end{proof}

For those 2-self-centered graphs that have connected complements, Proposition \ref{bipartite} is not useful. So, we may develop the characterization in some separate propositions for them, or, we can prove a more general statement which covers this case as a special case. In this paper, we do the later one, for which some preliminaries are needed.

\begin{definition}
Let $G$ be a 2-self-centered graph. A vertex $x$ in $G$ is called \textit{critical for $u$ and $v$} if $uv\notin E$ and $x$ is the only common neighbour of $u$ and $v$.
\end{definition}

\begin{lemma}
\label{tri-free}
Let $G$ be an edge minimal 2-self-centered graph with no critical vertex for any pair of vertices. Then $G$ is triangle-free. Furthermore, every triangle-free 2-self-centered graph is edge-minimal. 
\end{lemma}
\begin{proof}
Suppose in contrary, that there are $u,v,w \in V(G)$ such that $uv,vw,wu\in E(G)$. If $\deg(u)=\deg(v)=\deg(w)=2$ then $G$ is itself a triangle which contradicts to ${\rm rad}(G)=2$. If $\deg(u)=\deg(v)=2$ then ${\rm diam}(G)=2$ implies that all other vertices of $G$ is a neighbour of $w$. Thus $\deg(w)=n-1$ which contradicts the fact that ${\rm rad}(G)=2$. Hence, at most one of the vertices $u,v$ and $w$ has degree 2. Suppose that $\deg(u)$ and $\deg(v)$ are both greater than 2.

Let $e=uv$ and $H=G\setminus e$. Then edge-minimality of $G$ implies that $H$ is not a 2-self-centered graph. Since $\deg_{G}(u)$ and $\deg_{G}(v)$ are at least 3, this happens only if there are two vertices $x$ and $y$ such that $d_H(x,y)>2$. Since $d_G(x,y)\leqslant2$ it can be deduced that $\{x,y\}\cap\{u,v\}\neq\emptyset$. The cases $x=u$ and $y=v$ cannot happen at the same time because we have the path $x\sim w\sim y$ in $H$. If $x=u$ and $y$ is a vertex other than $v$, then there is a path $xt,ty$ in $G$ for some vertex $t$, since $v$ is not critical for $u$ and $y$. For the case $y=v$ and any other vertex $x$ the argument is similar. Therefore, $H$ is a 2-self-centered graph, a contradiction.

Moreover, let $G$ be triangle-free. If $G$ is not edge-minimal then there is an edge $e$ with ends $u$ and $v$ such that $G\setminus e$ is still a 2-self-centered graph. Thus there is a path of length 2 between $u$ and $v$ in $G\setminus e$. This gives a triangle in $G$. 
\end{proof}

Nevertheless, there are examples of edge-minimal 2-self-centered graphs possessing some critical vertices \textit{with} or \textit{without} triangles. 

\begin{example}
\label{ex1}
Let $G$ be the graph with vertex set $V=\{0,1,2,3,4,5,6,7\}$ and edge set $E=\{01,23,12,14,15,23,36,37,46,57,67\}$. Then $G$ is an edge-minimal 2-self-centered graph possessing the critical vertex $6$ for the vertices $4$ and $7$, \textit{with} a triangle on $3,6,7$, see figure 1.
\end{example}

\begin{figure}[htb]
\begin{center}
\includegraphics[scale=0.65]{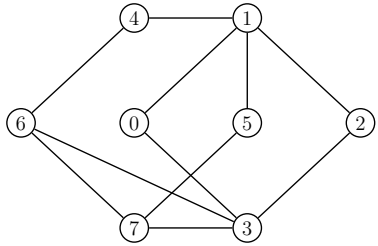}

\caption{ The graph $G$ of Example \ref{ex1}.}
\end{center}
\end{figure}

\begin{example}
\label{ex2}
Let $H$ be a graph constructed in the following way: consider the graph $K_{3,3}$ with two vertices $y$ and $z$ in different parts connected by the edge $e$. Omit $e$ and add a vertex $x$ with two edges $xy$ and $xz$ to obtain a graph $G$. Then $G$ is an edge minimal 2-self-centered graph possessing the critical vertex $x$ for the vertices $y$ and $z$, \textit{without} any triangle.
\end{example}

\begin{figure}[htb]
\begin{center}
\includegraphics[scale=0.65]{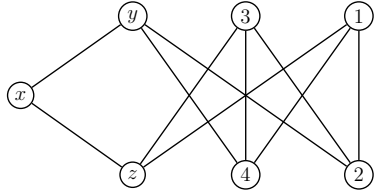}

\caption{The graph $H$ of Example \ref{ex2}.}
\end{center}
\end{figure}

\begin{definition}
\label{GBIC}
A graph $G$ is called to have a \emph{Specialized Bi-Independent Covering via $(\mathbb{A}_r,\mathbb{B}_s)$} if
\begin{itemize}
\item[(i)] $G$ is triangle-free,
\item[(ii)] there are two families $\mathbb{A}_{r}= \{ A_{1},\ldots,A_{r} \}$ and $\mathbb{B}_{s}= \{ B_{1}, \ldots B_{s}\}$ of not necessarily distinct independent subsets of $G$ such that we have $V(G) =\cup_{i=1}^r A_{i}=\cup_{j=1}^s B_{j}$,
\item[(iii)] for all $u,v \in V(G)$ if $d(u,v) \geq 3$ then there is an $1\leq i \leq r$ such that $u,v \in A_{i}$ or there is $1\leq j\leq s$ such that $u,v \in B_{j}$,
\item[(iv)] for all $u \in V(G)$ and $i\in \{ 1, \ldots, r\}$ if $d(u,A_{i})\geq 2$ then there is a $j\in \{ 1, \ldots, s\}$ such that $A_{i} \cap B_{j} = \emptyset$ and $u\in B_{j}$, and
\item[(v)] for all $u \in V(G)$ and $j\in \{ 1, \ldots, s\}$ if $d(u,B_{j})\geq 2$ then there is an $i\in \{ 1, \ldots, r\}$ such that $A_{i} \cap B_{j} = \emptyset$ and $u\in A_{i}$.
\end{itemize}
\end{definition}

To make it easy, we shorten the name ``specialized bi-independent covering'' to \emph{SBIC}. It is straightforward to check that every triangle-free graph $G$ has two families of independent sets $\mathbb{A}_r,\mathbb{B}_s$ such that $G$ has a SBIC via $(\mathbb{A}_{r},\mathbb{B}_{s})$. To see this, fix two independent coverings of $G$, and  by adding enough independent sets to them, we can always satisfy items (iii) to (v) of Definition \ref{GBIC}.

We need the following definition to complete our characterization of triangle-free 2-self-centered graphs.
\begin{definition}
\label{GCB}
A graph $G$ is called an \emph{$X$-generalized complete bipartite}, denoted by ${\rm GCB}_{X}(k,\ell,\mathbb{A}_r,\mathbb{B}_s)$, if $X$ has an SBIC via $(\mathbb{A}_r,\mathbb{B}_s)$ and $G$ is constructed in the following way:
\begin{itemize}
\item[(1)] $V(G)=K \cup L \cup Y \cup Z \cup V(X)$ where $\vert K \vert =k$, $\vert L \vert =\ell$, $Y=\{ y_{1},\ldots, y_{r} \}$ and $Z =\{ z_{1},\ldots, z_{s} \}$.
\item[(2)] $a \sim t$ for all $a\in K$ and $t \in L \cup Y$.
\item[(3)] $b \sim t$ for all $b\in L$ and $t \in K \cup Z$.
\item[(4)] $y_{i} \sim t$ for all $t \in A_{i}$ and $1\leq i\leq r$.
\item[(5)] $z_{j} \sim t$ for all $t \in B_{j}$ and $1\leq j\leq s$.
\item[(6)] $y_{i} \sim z_{j}$ if and only if $A_{i} \cap B_{j} = \emptyset$.
\end{itemize}

Moreover, there are some special cases that must be treated separately:

\begin{itemize}
\item[(7)] If $k=0$ then every member of $Y$ has a neighbour in $Z$ and for all $i,j \in \{ 1, \ldots , r \}$ we have $A_{i} \cap A_{j} \neq \emptyset$ or there is a $p\in \{ 1, \ldots, s\}$ such that $A_{i} \cap B_{p} = A_{j} \cap B_{p} =\emptyset$.

\item[(8)] If $\ell=0$ then every member of $Z$ has a neighbour in $Y$ and for all $i,j \in \{ 1, \ldots , r \}$ we have $A_{i} \cap A_{j} \neq \emptyset$ or there is a $p\in \{ 1, \ldots, s\}$ such that $A_{i} \cap B_{p} = A_{j} \cap B_{p} =\emptyset$.

\item[(9)] If $r= 0$ then $k\neq 0$ and if $s=0$ then $\ell \neq 0$.

\item[(10)] $r=s=0$ if and only if $X=\emptyset$ and $k,\ell \geq 2$.

\item[(11)] If $\vert X \vert =1$ then at least one of $k$ or $\ell$ is non-zero.
\end{itemize}

\end{definition}
\begin{proposition}\label{GBPTriFree}
Any generalized complete bipartite graph is a triangle-free 2-self-centered graph.
\end{proposition}
\begin{proof}
Let $G= {\rm GCB}_{X}(k,\ell,\mathbb{A}_r,\mathbb{B}_s)$ and $t=|X|$. Then $n:=|V(G)|=k+\ell+r+s+t$. We show that $G$ has no vertex of degree $n-1$ and then we show that item (ii) of Lemma \ref{2self} holds for $G$. Then by Remark \ref{remark} we deduce that $G$ is 2-self-centered. By proving that $G$ has no triangle and using Lemma \ref{tri-free}, we actually show that $G$ is also edge-minimal.  

For $a\in K$, $\deg(a)=\ell+r=n-s-k-t$. Thus if $r=s=0$ then by item (10) of Definition \ref{GCB} we have $k\geq 2$ and hence $\deg(a)\leqslant n-2$. If $r$ or $s$ is non-zero, then by item (10) we have $t \neq 0$ and therefore we have $\deg(a)\leqslant n-2$ (by items (2) and (9) of Definition \ref{GCB}, and because no element of $K$ is adjacent to a vertex of $X$). 

For $b\in L$, by a similar proof to the case $a \in K$ we can deduce that $\deg(b)\leqslant n-2$.

For $y_i\in Y$, if $\ell \neq 0$ then $\deg(y_i)\leqslant n-2$ because no element of $L$ is adjacent to $y_i$. If $\ell = 0$ then either $y_i$ is not adjacent to all vertices of $X$ or if $y_i$ is adjacent to all vertices of $x$ then it is not adjacent to $z_j$ for some $j \in \{ 1, \ldots, s \}$ (which its existence is supported by item (8) of Definition \ref{GCB}), each of which cases that yields to $\deg(y_i)\leqslant n-2$.

For $z_j\in Z$ we have the same argument to $y_i\in Y$. 

Finally, for each $x\in X$, item (11) of Definition \ref{GCB} guarantees that $\deg(x) \leqslant n-2$ whenever $X$ has only one vertex. So, assume that $t \geqslant 2$. Therefore, There are two possibilities: either there is $\hat{x} \in X$ such that $x$ is not adjacent to $\hat{x}$, or, $x$ is adjacent to all other vertices of $X$. If the former case is true then $\deg(x) \leqslant n-2$. For the later case, since $x$ is not in any independent set with other vertices of $X$, we have there is some $i \in \{ 1,\ldots,r \}$ and $j \in \{ 1,\ldots,s \}$ such that $\{ x \} \cap A_{i}=\{ x \} \cap B_{j}=\emptyset$. Thus, by items (4) and (5) of Definition \ref{GCB}, we have $x$ is not adjacent to $y_i$ and $z_j$ and hence $\deg(x) \leqslant n-2$.

To show that (ii) of  Lemma \ref{2self} is also satisfied, we should choose two vertices $u$ and $v$ in $G$ and show that whenever they are not adjacent, they have at least one common neighbour. There are 15 different ways for choosing $u$ and $v$ from $G=K\cup L \cup Y \cup Z \cup X$.

If $(u,v)\in (K\times L)\cup(K\times Y)\cup(L\times Z)$ then $u$ and $v$ are adjacent to each other.

If $(u,v)\in (K\times K)\cup(L\times L)$ then there is a path of length 2 between $u$ and $v$ via one of the sets $L \cup Y$ or $K \cup Z$.

If $(u,v) \in (Y\times Y)$ then if $k \neq 0$ there is a path of length 2 between $u$ and $v$ via any member of $K$. If $k=0$ then, by item (7) of Definition \ref{GCB}, we have either there is a $z_p \in Z$ which is a common neighbour of $u$ and $v$, or, $u$ and $v$ are both adjacent to a vertex $x \in X$. The case $(u,v) \in (Z\times Z)$ is also similar.

If $(u,v) \in (L\times Y)$ then if $k\neq 0$ there is a path of length 2 between $u$ and $v$ via any member of $K$. If $k=0$ then, by item (7) of Definition \ref{GCB}, we have every member of $Y$ has a neighbour in $Z$, so $Z$ is non-empty and $v$ has a neighbour in $Z$, namely $\hat{z}$. Since $u$ is also adjacent to $\hat{z}$ by item (3) of Definition \ref{GCB}, there is a path of length 2 between $u$ and $v$. The case $(u,v) \in (K\times Z)$ is also similar.

If $(u,v)\in Y\times Z$ then $u=y_i$ and $v=z_j$ for some $i$ and $j$. If $A_i\cap B_j\neq\emptyset$ then we can choose a $c$ in $A_i\cap B_j$ such that there is a path of length 2 between $u$ and $v$ via $c$. If $A_i\cap B_j=\emptyset$ then $u$ is adjacent to $v$, by item (6) of Definition \ref{GCB}.

If $(u,v)\in X\times X$ then either $d_{X}(u,v)\leqslant 2$ or by item (iii) of Definition \ref{GBIC} there is an $i \in \{ 1,\ldots,r \}$ or a $j \in \{ 1,\ldots, s \}$ such that both $u$ and $v$ are adjacent to $y_i$ or $z_j$.

If $(u,v)\in (K\times X)\cup (L\times X)$ then there is an $i \in \{ 1,\ldots,r \}$ or a $j \in \{ 1,\ldots, s \}$ such that $v$ is adjacent to $y_i$ and $z_j$. Then, since $u$ is adjacent to $y_i$ or $z_j$, we have $d_{G}(u,v)=2$.

If $(u,v)\in (Y\times X)$ then then there is an $i \in \{ 1,\ldots,r \}$ such that $u=y_i$. Then, either $d(v,A_{i})\leqslant 1$ which means that $d(u,v)\leqslant 2$, or, if $d(v,A_{i})\geqslant 2$ then by item (iv) of Definition \ref{GBIC} there is a $j\in \{ 1, \ldots, s\}$ such that $A_{i} \cap B_{j} = \emptyset$ and $v\in B_{j}$. Hence by items (5) and (6) of Definition \ref{GCB} we have $z_j$ is adjacent to both $u$ and $v$. The case $(u,v)\in (Z\times X)$ is also similar.

So, for each of 15 ways of choosing $u$ and $v$ from vertices of $G$ we have $d(u,v)\leqslant 2$.

We finally show that $G$ is triangle-free. On contrary, suppose that $u,v,w$ are vertices of a triangle in $G$. The case that none of $u,v$ and $w$ is a vertex of $X$ cannot happen because $K \cup Z$ and $L \cup Y$ are independent sets. Since $X$ is triangle-free, $u,v$ and $w$ are not all together vertices of $X$. Meanwhile, if only two vertices of $\{ u,v,w \}$ are in $X$, then the third is not adjacent to the other two because they cannot be in the same independent set in $X$. So, at most one of $\{ u,v,w \}$ is a vertex of $X$. Let for instance $w$ be a vertex of $X$. Then $u$ and $v$ are not members of $Y$ or $Z$ at the same time, because otherwise they are not adjacent together. The case that one of $u$ and $v$ is in $Y$ and the other in $Z$ is also impossible because it is contrary to item (6) of Definition \ref{GCB}.
\end{proof}
\begin{theorem}\label{TriFree2SC}
A graph $G$ is a triangle-free 2-self-centered graph if and only if there are positive integers $k,\ell,r,s$ and a graph $X$ which has a SBIC via $(\mathbb{A}_r,\mathbb{B}_s)$ such that $G={\rm GCB}_{X}(k,\ell,\mathbb{A}_r,\mathbb{B}_s)$.
\end{theorem}

\begin{proof}
Let $Y'$ be a maximal independent subset of $G$, let $Z'$ be a maximal independent subset of $G\setminus Y'$ and let $X=G\setminus(Y'\cup Z')$. Suppose that $K$ (resp. $L$) is the set of all vertices in $Z'$ (resp. $Y'$) which are not adjacent to any member of $X$ and put $Y=Y'\setminus L, Z=Z'\setminus K$. 

Let $a\in K$ and $y'\in Y'$. We claim that $ay'\in E$. Suppose on the contrary that $ay'\notin E$. Since ${\rm diam}(G)=2$ there is a $u$ in $G$ such that $au,uy'\in E$. The vertex $u$ cannot be in $Y'$ or $Z'$ since $Y'$ and $Z'$ are independent sets. Hence $u\in X$. This contradicts to the definition of $K$. 

A similar argument shows that each member of $L$ is adjacent to each member of $Z'$. 

Let $k=|K|, \ell=|L|, r=|Y|, s=|Z|, Y=\{y_1,\ldots,y_r\}$ and $Z=\{z_1,\ldots,z_s\}$. Now put $A_i=N_{X}(y_i)$ and $B_j=N_{X}(z_j)$. We show that $A_i$'s and $B_j$'s are independent subsets of $X$ and $X$ has a SBIC via $(\mathbb{A}_r,\mathbb{B}_s)$.

Let $x$ be an arbitrary member of $X$. Since $Y'$ and $Z'$ are maximal independent, there should be neighbours for $x$ in $Y'$ and $Z'$. We know that these neighbours are in $Y$ and $Z$. Let $y_i$ and $z_j$ be adjacent to $x$. Thus $x\in A_i$ and $x\in B_j$. This shows that $X=\cup_{i=1}^r A_i=\cup_{j=1}^s B_j$. 

Each $A_i$ and each $B_j$ is independent, since $G$ is triangle-free. Moreover, if $y_i$ and $z_j$ are not adjacent to each other, then since ${\rm diam}(G)=2$, there should be an $x\in X$ with $y_ix,xz_j\in E(X)$. Thus $x\in A_i\cap B_j$. If $y_i$ is adjacent to $z_j$ then there must not be such an $x$, so we have $y_{i} \sim z_{j}$ if and only if $A_{i} \cap B_{j} = \emptyset$.

Furthermore, $X$ is triangle-free since $X$ is a subgrpah of the triangle-free graph $G$.

Items (iii), (iv) and (v) of Definition \ref{GBIC} must holds because $G$ is a triangle-free 2-self-centered graph. Therefore $X$ has an SBIC via $(\mathbb{A}_r,\mathbb{B}_s)$.

Items (1) to (6) of Definition \ref{GCB} have already hold. Moreover, items (7) to (11) of Definition \ref{GCB} must also hold because $G$ is a triangle-free 2-self-centered graph. Hence $G = {\rm GCB}_{X}(k,\ell,\mathbb{A}_r,\mathbb{B}_s)$.

Since the converse is evident by Proposition \ref{GBPTriFree}, we are done with the proof.
\end{proof}

The reader should note that every complete bipartite graph $K_{k,\ell}$ with $k,\ell \geqslant 2$ is a generalized complete bipartite graph ${\rm GCB}_{\emptyset}(k,\ell,\emptyset,\emptyset)$.

Now, we can consider edge-minimal 2-self-centerd graphs with some triangles. We need the following procedure to proceed.

\begin{procedure}
\label{star}
Let $G$ be a graph, $u,v,w$ form a triangle in $G$ and suppose that $v$ is a critical vertex for $u$ and $v_{1}, \ldots, v_{p}$ and/or $u$ is a critical vertex for $v$ and $u_{1}, \ldots, u_{q}$. Remove the edge $uv$ and add edges $uv_{1}, \ldots, uv_{p}$ and $vu_{1}, \ldots, vu_{q}$.
\end{procedure}

The following theorem characterizes edge-minimal 2-self-centered graphs with triangles, which completes the characterization of all 2-self-centered graphs.

\begin{theorem}
Let $G$ be a graph. Then $G$ is an edge-minimal 2-self-centered graph with some triangle if and only if the following two conditions are true:
\begin{itemize}
\item[(i)] for each edge of every triangle in $G$, at least one end-vertex is a critical vertex (for the other end-vertex of that edge and some other vertices of $G$), and
\item[(ii)] iteration of Procedure \ref{star} on $G$ (at most to the number of triangles of $G$) transforms $G$ to a triangle-free 2-self-centered graph.
\end{itemize}
\end{theorem}

\begin{proof}
Assume that $u,v,w$ form a triangle in $G$. Since $G$ is edge-minimal, if we omit the edge $uv$ then the resulting graph is not 2-self-centered. This shows that $u$ or $v$ is a critical vertex. Let $u$ be a critical vertex. Thus there are vertices $u_1,\ldots, u_q$ such that $u$ is the common neighbour of $v$ and each of the $u_i$'s. Moreover, if $v$ is also a critical vertex for $u$ and some other vertices, then we suppose that $v_1,\ldots,v_p$ are the vertices such that $v$ is a common neighbour of $u$ and each of the $v_j$'s.

If we omit $uv$ and add edges $u_1{}v,\ldots,u_{q}v,uv_{1},\ldots,uv_{p}$ then the resulting graph $G'$ is clearly 2-self-centered and the number of triangles of $G'$ is less than the number of triangles of $G$. To see this, note that edges of a triangle on $u,v$ and $w$ are omitted and no new triangle is added. In contrary, suppose that we have a new triangle. Then it should be of the form $u_i,v,t$ (or $v_j,u,s$) which contradicts to the fact that $u$ (or $v$) is a critical vertex for $u_i$ and $v$ (for $v_j$ and $u$).

If $G$ has still some triangle then we can proceed this process. Therefore, we finally transform $G$ into a triangle-free 2-self-centered graph.

Conversely, if the two conditions are true for a graph $G$ with some triangles, then $G$ is an edge-minimal 2-self-centered graph because condition (ii) guarantees that $G$ is 2-self-centered while condition (i) obligates $G$ to be edge-minimal.
\end{proof}

\begin{flushright}
Received 30 November 2015\\
Revised 24 August 2016\\
Accepted 1 September 2016
\end{flushright}

\end{document}